\documentclass{article}

\usepackage{color}
\usepackage{amsthm,amssymb,amsfonts,mathrsfs}
\usepackage{amsmath}

\newcommand{\RN}[1]{%
	\textup{\uppercase\expandafter{\romannumeral#1}}%
}

\newtheorem{Theorem}{Theorem}[section]
\newtheorem{Proposition}{Proposition}[section]
\newtheorem{Lemma}{Lemma}[section]

\theoremstyle{definition}

\newtheorem{Definition}{Definition}[section]

\newtheorem{Assumptions}{Hypothesis}[section]

\def\R{{\mathbb{R}}}

\def\cU{{\mathcal{U}}}

\def\t\cU{{\widetilde{{\mathcal{U}}}}}

\newcommand\norm[1]{\left\lVert#1\right\rVert}
\newcommand{\restr}[1]{|_{#1}}

\def\ds{\displaystyle}

\title {Fourth order differential operators with interior degeneracy and generalized Wentzell boundary conditions}
\author{{{\sc Alessandro Camasta}\thanks{The author is a member of the  {\it Gruppo Nazionale per l'Analisi Ma\-te\-matica, la Probabilit\`a e le loro Applicazioni (GNAMPA)} of the Istituto Nazionale di Alta Matematica (INdAM) and a member of {\it UMI ``Modellistica Socio-Epidemiologica (MSE)''}. He is partially supported by PRIN 2017-2019 {\it Qualitative and quantitative aspects of nonlinear PDEs.}}}\\
	Department of Mathematics\\ University of Bari Aldo Moro\\
	Via
	E. Orabona 4\\ 70125 Bari - Italy\\ e-mail: alessandro.camasta@uniba.it\\
	{\sc Genni Fragnelli}\thanks{The author is a member of the  {\it Gruppo Nazionale per l'Analisi Ma\-te\-matica, la Probabilit\`a e le loro Applicazioni (GNAMPA)} of the Istituto Nazionale di Alta Matematica (INdAM) and a member of {\it UMI ``Modellistica Socio-Epidemiologica (MSE)''}. She is supported partially by the FFABR {\it Fondo per il finanziamento delle attivit\`a base di ricerca} 2017, by PRIN 2017-2019 {\it Qualitative and quantitative aspects of nonlinear PDEs} and by the DEB.HORIZON$_{-}$EU$_{-}$DM737 project 2022{\it COntrollability of PDES in the Applied Sciences (COPS).}}\\
	Department of Ecology and Biology\\ Tuscia University\\
	Largo dell'Universit\`a, 01100 Viterbo - Italy\\ e-mail: genni.fragnelli@unitus.it}

\date{}
\begin{document}

\maketitle

\vspace{0.3cm}

\centerline{ {\it  }}

\begin{abstract}
In this paper we consider the fourth order operators $A_1u:=(au'')''$ and $A_2u:=au''''$ in divergence and non divergence form, respectively, where $a:[0,1]\to\mathbb{R}_+$  degenerates in an interior point of the interval. Using the semigroup technique, under suitable assumptions on $a$, we study the generation property of these operators associated to  generalized Wentzell boundary conditions, proving the well posedness of the corresponding parabolic problems.
\end{abstract}


\noindent 2000AMS Subject Classification: Primary 47D06, 35K65; Secondary 47B25, 47N20

\section{Introduction}
In Mathematical Analysis the boundary conditions associated with a differential operator usually involve the function and its derivatives (including Dirichlet, Neumann and Robin conditions). In some cases, as in Markov Process Theory, it is natural to include boundary conditions involving the operator itself (see \cite{7} or \cite{wentzell}, for a detailed exposition). 
In particular, if $A$ denotes an elliptic operator, the parabolic problem
\begin{equation*}
	\frac{\partial u}{\partial t}-Au=0\quad\,\,\,\,\,\,\,\text{ in } \Omega\subset\mathbb{R}^n, t\ge 0,
\end{equation*}
is said to be equipped with the \textit{Wentzell boundary condition} if one demands $Au(t,x)=0$ for $x\in\partial\Omega$ and any $t\ge 0$.
In the literature a more general boundary condition which arises naturally in the context of the heat equation is the \textit{generalized Wentzell boundary condition} (GWBC)
\begin{equation}\label{general wentzell BC}
	\alpha Au(x)+\beta\frac{\partial u}{\partial n}(x)+\gamma u(x)=0,\quad\,\,\,\,\,\,\,x\in\partial\Omega,
\end{equation}
where $(\alpha,\beta,\gamma)\neq (0,0,0)$. Note that $(\alpha,\beta,\gamma)$ can depend on $x$ as well.
Surprisingly, these boundary conditions arise naturally as part of the formulation of the problem and are incorporated in the derivation of the heat equation itself. 
Additional motivation for the study of evolution equations with (GWBC) comes from their possible interpretation as evolution equations with dynamical boundary conditions (for a general view on the role of Wentzell boundary conditions we refer to \cite{coclite}). It is worth to mention that, in the case of heat equations, (GWBC) allow to take into account the action and the effect of heat sources on the boundary (see \cite{7}).
For a systematic study of the derivations and physical interpretations of Wentzell boundary conditions we refer, e.g., to \cite{7}, which covers heat and wave equations. On the other hand, for beam equations, Cahn-Hilliard equations and related models one can see, e.g., \cite{gal}, \cite{gal 1}, \cite{gal 2} and the references therein.
Due to their importance and physical interpretation, let us mention briefly other mathematical problems and contexts in which these boundary conditions appear. 
For example, if $\Omega$ is a smooth bounded domain in $\mathbb{R}^n$ and $B$ is a formally symmetric differential operator on $L^2(\Omega)$ with domain $D(B)\subset \mathcal{C}_0^{\infty}(\Omega)$, a classical problem is to find all self-adjoint extensions of $B$ (or all self-adjoint restrictions of $B^*$). This problem is solved abstractly by von Neumann and has been worked out in detail in some concrete cases. For example, if $\Omega =(0,1)$ and $B=\frac{1}{i}\frac{d}{dx}$, the self-adjoint extensions of $B$ are determined by boundary conditions and are parametrized by the unit circle in $\mathbb{C}$.
An analogous problem is to consider $B=\Delta$ with $D(B)=\mathcal{C}_0^{\infty}(\Omega)$.
 On the space $\mathcal{C}_0(\Omega)$, in which $\mathcal{C}_0^{\infty}(\Omega)$ is dense, $\overline{B}$ generates a positive contraction semigroup and an interesting question is to understand which extensions of $B$ on $\mathcal{C}(\overline{\Omega})$ have this property. In his pioneering work \cite{wentzell} Wentzell shows that such extensions are characterized by (Wentzell) boundary conditions of the form (\ref{general wentzell BC}), where $(\alpha(x),\beta(x),\gamma(x))\neq (0,0,0)$ for all $x\in\partial\Omega$ with $\alpha>0$, $\beta\ge 0$ and $\gamma\ge 0$ (Wentzell's work generalizes previous results by W. Feller in one space dimension, see, e.g., \cite{feller}).
Furthermore, other multiple applications are possible. For example, in \cite{vogt}  the authors consider a Dirichlet problem that describes the basic diffusion of particles in a locally compact space, endowed with a Radon measure. Moreover, in a very general setting, they present an abstract version of the Wentzell boundary conditions (see also \cite{mugnolo}).
In the context of Classical and Quantum Field Theory on asymptotically anti-de Sitter spacetimes (AdS) and other spacetimes with boundaries, \cite{dappiaggi} studies a massive scalar field in AdS in $d+1$ spacetime dimensions subject to (GWBC) and it is highlighted that they are dynamical boundary conditions invariant under the action of the isometry group of the AdS boundary.
The treatment of Wentzell boundary conditions in the classical and quantum field theoretic literature appears also in \cite{barbero}, where the classical mechanical system of a finite string with point masses subject to harmonic potentials in the extrema is solved. (GWBC) are also considered in \cite{zhan} in $(d+1)$-dimensional Minkowski spacetime with one or two timelike boundaries. In particular, in \cite{zhan} the author shows that the Wentzell boundary conditions ensure that the short-distance singularities of the two-point function for the boundary field has the form expected of a field living in a $d$-dimensional spacetime, contrary to other boundary conditions, for which the two-point function inherits the short-singularity of the $(d+1)$-dimensional bulk. This seems to be a very desired feature for holographic purposes.
In Biology a significant amount of interest has been devoted to the analysis of mathematical models arising in structured population dynamics and a very important problem is the choice of suitable boundary conditions for a biologically plausible and mathematically sound model.
In this field \cite{popol calsina} introduces and analyzes a structured population model, with so called distributed recruitment term and (GWBC), describing the dynamics of a population infected with a certain type of bacteria. 
We refer to \cite{popol farkas} and \cite{popol mokhtar richard} for a model of structured populations with generalized Wentzell-Robin boundary conditions. On the other hand, for more general operators one can see  \cite{popol calsina}, \cite{popol farkas}, \cite{popol mokhtar} and \cite{popol mokhtar richard}.
Finally, it is well known that degenerate parabolic equations are widely used as mathematical models in the applied sciences to describe the evolution in time of a given system. For this reason, in recent years an increasing interest has been devoted to the study of differential degenerate operators in divergence or in non divergence form. In particular,  after the new directions opened in \cite{2}, \cite{4}, great attention is given to the operators
\[
\mathcal A_1 u:= (au')'
\]
and
\[
\mathcal A_2 u:= au''
\]
with general Wentzell boundary conditions. 

Actually, in this paper we are interested in \textit{fourth order operators} since many problems that are relevant for applications are described by these operators. Among these applications we can find dealloying (corrosion processes), population dynamics, bacterial films, thin film, tumor growth, clustering of mussels and so on (see \cite[Introduction]{CF} for some detailed references). Moreover, operators of this type with suitable domains involving different boundary conditions arise in a natural way in several contexts as beam analysis and Euler-Bernoulli beam theory (see \cite[Introduction]{CF_Neumann} for some related references in this field).

The novelty of this paper is that we prove the generation property for the following \textit{degenerate} fourth order operators 
\[
A_1 u:= (au'')''
\]
and
\[
A_2 u:= au''''
\] 
equipped with (GWBC) (see the main theorems in Sections 3-4), obtaining the existence of solutions for the associated parabolic Cauchy problems. Here $a\in \mathcal C[0,1]$ is such that there exists $x_0\in (0,1)$ such that $a(x_0)=0$ and $a(x)>0$ if $x \neq x_0$.
As far as we know, this is the first paper that deals with this problem; thus we intend to fill this gap following the ideas of  \cite{favini vol spec} and \cite{4}.

The paper is organized as follows: in Section \ref{section 2} we prove some preliminary results that hold if $a$ degenerates in a general point $x_0 \in [0,1]$. In Section \ref{sezione 3} and in Section \ref{sezione 4} we assume that the degeneracy point $x_0$ belongs to $(0,1)$ and we prove that the operators $A_i$, $i=1,2$, equipped with (GWBC) are non negative and self-adjoint with dense domain, obtaining the well posedness for the associated Cauchy problems.
It is worth noting that in this paper we deal with real function spaces, but the assertions can be easily extended to the complex case.

\noindent Notations:\\
 $C$ denotes universal positive constants which are allowed to vary from line to line;\\
$'$ denotes the derivative of a function depending on the real space variable $x$.

\section{Preliminary results}\label{section 2}
In this section we recall some suitable weighted spaces and preliminary results given in \cite{CF_Neumann}, that will be crucial for the rest of the paper.
\newline
\indent
For simplicity, we distinguish between the case of a weakly degenerate function and of a strongly degenerate one.
 \subsection{The weakly degenerate case}\label{sezione 2.1}

First of all, we give the following definition on the function $a$.
\begin{Definition}
	A function $a\in \mathcal{C}[0,1]$ is said to be weakly degenerate if there exists $x_0\in [0,1]$ such that $a(x_0)=0$, $a(x)>0$ for all  $x \in [0,1]\setminus \{x_0\}$ and $\displaystyle\frac{1}{a}\in L^1(0,1)$.
\end{Definition} 
As an example of a weakly degenerate function we can take $a(x)=|x-x_0|^K$, with $0<K<1$.
\newline
\indent
In order to deal with the divergence case, for any weakly degenerate function $a\in \mathcal{C}[0,1]$, let us introduce the following weighted spaces:
\begin{equation}\label{(2.9) primo}
	\begin{aligned}
		H^i_a(0,1):=\{u\in H^{i-1}(0,1):& \; u^{(i-1)} \text{ is absolutely continuous in [0,1]},\\
		& \sqrt{a}u^{(i)}\in L^2(0,1)\},
	\end{aligned}
\end{equation}
endowed with the norms
\begin{equation}\label{normadiv}
\|u\|^2_{H^i_a(0,1)}:= \sum_{j=0}^{i-1}\|u^{(j)}\|^2_{L^2(0,1)}+ \|\sqrt{a}u^{(i)}\|^2_{L^2(0,1)} \quad \forall \; u \in H^i_a(0,1),
\end{equation}
$i=1,2;$ here  $H^0(0,1):= L^2(0,1)$ and $u^{(0)}=u$.
The following proposition holds.
\begin{Proposition}\label{equivalenza1} For  all $u \in H^i_a(0,1)$, $i=1,2$, set
\[
\|u \|_{i,a}^2:= \|u\|^2_{L^2(0,1)}+ \|\sqrt{a}u^{(i)}\|^2_{L^2(0,1)}.
\]
Then $\|u\|_{H^i_a(0,1)}$ and $\|u \|_{i,a}$ are equivalent.
\end{Proposition}
To prove this proposition, the next result is essential.
\begin{Lemma}[\cite{leoni}, Theorem 7.37]\label{leoni}
Let $I=(a,b)$, with $a,b\in\mathbb{R}$, $a<b$, and let $1 \le p,q,r \le +\infty$ be such that
\[
\frac{1}{2q}+ \frac{1}{2p} \ge \frac{1}{r}. 
\]
Let $u \in W^{2,1}_{\text{loc}}(I)$, then there exists $c=c(p,q,r)>0$ such that
\begin{equation}\label{7.11}
\|u'\|_{L^r(I)} \le cl^{\frac{1}{r}-1-\frac{1}{q}}\|u\|_{L^q(I)}+ c l^{1-\frac{1}{p}+\frac{1}{r}}\|u''\|_{L^p(I)}
\end{equation}
for every $0<l<\mathcal L^1(I)$, where $\mathcal L^1(I)$ is the one-dimensional measure of $I$.
\end{Lemma}
\begin{proof}[Proof of Proposition \ref{equivalenza1}]
For $i=1$ the thesis is obvious. Now, take $i=2$ and $u \in H^2_a(0,1)$. Clearly,
\[
\|u\|^2_{2,a} \le \|u\|^2_{H^2_a(0,1)}.
\]
For the other estimate, it is sufficient to prove
\begin{equation}\label{stima_der}
\|u'\|^2_{L^2(0,1)} \le C\left( \|u\|_{L^2(0,1)}^2+\|\sqrt{a}u''\|^2_{L^2(0,1)}\right),
\end{equation}
for a positive constant $C$. To this aim, observe  that $u\in W^{2,1}_{\text{loc}}(0,1)$. Thus, by Lemma \ref{leoni} with $I=(0,1)$, $p=1$, $q=2$ and $r=2$, one has
\begin{equation}\label{dis1}
\|u'\|_{L^2(0,1)} \le cl^{-1} \|u\|_{L^2(0,1)}+ cl^{\frac{1}{2}}\|u''\|_{L^1(0,1)}.
\end{equation}
Now, by the H\"older inequality,
\[
\begin{aligned}
\|u''\|_{L^1(0,1)}&= \int_0^1 \frac{\sqrt{a}|u''|}{\sqrt{a}}dx\le \left( \int_0^1 (a(u'')^2)(x)dx \right)^{\frac{1}{2}}\left(\int_0^1\frac{1}{a(x)}dx\right)^{\frac{1}{2}}\\
&= \|\sqrt{a}u''\|_{L^2(0,1)}\left\| \frac{1}{a}\right\|_{L^1(0,1)}^{\frac{1}{2}}.
\end{aligned}
\]
Using this inequality in \eqref{dis1}, one has
\[
\|u'\|^2_{L^2(0,1)} \le C\left( \|u\|_{L^2(0,1)}^2 +\|\sqrt{a}u''\|^2_{L^2(0,1)} \right)
\]
for some suitable constant $C>0$ and the thesis follows.
\end{proof}

Moreover, observe that
	\begin{equation*}
	u \in H^{2}_{a}(0, 1) \Longrightarrow u'\in H^1_{a}(0, 1).
	\end{equation*}

Using the space $H^{2}_{a}(0, 1) $, we define
\begin{equation}\label{Z}
\mathcal Z_w(0,1):=\{u\in H^2_a(0,1) : au''\in H^2(0,1)\}.
\end{equation}
\vspace{0,3cm}

In order to deal with the non divergence case, in place of $L^2(0,1)$ and $H^i_a(0,1)$, we consider the spaces
\begin{equation*}
L^2_{\frac{1}{a}}(0, 1):=\biggl \{u\in L^2(0, 1):\int_{0}^{1}\frac{u^2}{a}\,dx<+\infty \biggr \}
\end{equation*}
and
\begin{equation*}
H^i_{\frac{1}{a}}(0, 1):= L^2_{\frac{1}{a}}(0, 1)\cap H^i(0, 1),
\end{equation*}
with the respective norms defined by
\begin{equation*}
\norm{u}^2_{L^2_{\frac{1}{a}}(0, 1)}:= \int_{0}^{1}\frac{u^2}{a}\,dx \quad\forall \;u \in L^2_{\frac{1}{a}}(0, 1),
\end{equation*}
and
\begin{equation}\label{norma}
\norm{u}_{H^i_{\frac{1}{a}}(0, 1)}^2:=\norm{u}^2_{L^2_{\frac{1}{a}}(0, 1)} + \sum_{j=1}^{i}\| u^{(j)}\|^2_{L^2(0, 1)} \quad \forall \;u \in H^i_{\frac{1}{a}}(0, 1),
\end{equation}
$i=1,2$.
Observe that for all $u \in H^i_{\frac{1}{a}}(0, 1)$, one can prove that $\|u\|^2_{H^i_{\frac{1}{a}}(0,1)}$ is equivalent to the following one
\[
\|u\|_{i,\frac{1}{a}}^2:= \|u\|^2_{L^2_{\frac{1}{a}}(0,1)} +  \|u^{(i)}\|^2_{L^2(0,1)}.
\]
For $i=1$ the equivalence is obvious; for $i=2$ it follows by \cite[Chapter VIII]{brezis}. Indeed, for all $u \in H^i(0,1)$, one has that $\|u\|_{H^i(0,1)}$ is equivalent to $\|u\|_{L^2(0,1)} + \|u^{(i)}\|_{L^2(0,1)}$, $i=1,2$. However,
$\ds \|u\|_{L^2(0,1)}^2= \int_0^1\frac{u^2(x)}{a(x)}a(x)dx \le C  \int_0^1\frac{u^2(x)}{a(x)}dx=C\|u\|^2_{L^2_{\frac{1}{a}}(0,1)}$ for all $u \in L^2_{\frac{1}{a}}(0,1)$ and, in particular, for all $u\in H^i_{\frac{1}{a}}(0, 1)$. Hence $\|u\|_{i,\frac{1}{a}}$ is equivalent to $\norm{u}_{H^i_{\frac{1}{a}}(0, 1)}$, for all $u \in H^i_{\frac{1}{a}}(0, 1)$ and for $i=1,2$.

Finally, consider the space
\begin{equation}\label{W}
\mathcal W_w(0,1):=\left\{u\in H^2_{\frac{1}{a}}(0, 1): au''''\in L^2_{\frac{1}{a}}(0, 1) \right\}.
\end{equation}
Clearly, it is a trivial fact that, if $au''''\in L^2_{\frac{1}{a}}(0, 1)$ then $u''''\in L^1(0,1)$ (since $\frac{1}{a} \in L^1(0,1)$) and, by \cite[Lemma 2.1]{CF}, $u\in W^{4,1}(0,1)$.

Using the same considerations of \cite[Proposition 3.1]{CF}, the spaces $H^i_{\frac{1}{a}}(0, 1)$ and $H^i(0, 1)$, $i=1,2$, coincide algebraically and the two norms are equivalent. Indeed, for all $u \in H^i(0,1)$ it is sufficient to prove that there exists a positive constant $C$ such that
	\[
	\|u\|^2_{L^2_{\frac{1}{a}}(0,1)} \le C\left\| \frac{1}{a}\right\|_{L^1(0,1)};
	\]
but this follows immediately taking $C:= \|u\|_{\mathcal C[0,1]}^2$.
\newline
Hence, if $u \in \mathcal W_w(0,1)$, then $u \in \mathcal{C}^3[0,1]$, and, in particular, $(au^{(k)})(x_0)=0$, for all $k=0,1,2,3$, being $a(x_0)=0$ and $u^{(k)} \in \mathcal C[0,1]$, for all $k=0,1,2,3$.

In the previous spaces the following Green formulas hold.
\begin{Lemma}[\cite{CF_Neumann}, Lemmas 2.1 and 3.1]\label{lemma 2.1}
	If $a$ is weakly degenerate, then
	\begin{itemize}
		\item[(i)] for all $(u,v)\in \mathcal Z_w (0,1)  \times H^2_a(0,1)$
		\begin{equation}\label{gf0}
		\int_{0}^{1}(au'')''v\,dx=[(au'')'v]^{x=1}_{x=0}-[au''v']^{x=1}_{x=0}+\int_{0}^{1}au''v''dx;
	    \end{equation}
		\item[(ii)] for all $(u,v)\in \mathcal W_w(0,1)\times H^2_{\frac{1}{a}}(0, 1)$
		\begin{equation}\label{gf2}
		\int_{0}^{1} u''''v\,dx=[u'''v]^{x=1}_{x=0}-[u''v']^{x=1}_{x=0}+\int_{0}^{1}u''v''dx.
		\end{equation}
	\end{itemize}
\end{Lemma}

We underline that in \cite{CF_Neumann}, to prove the previous lemma, the requirement $\frac{1}{a} \in L^1(0,1)$ is not used; actually it is sufficient to require $a\in \mathcal C[0,1]$.
\subsection{The strongly degenerate case}

In this subsection we consider another notion of degeneracy: the strongly one.
\begin{Definition}
		A function $a\in \mathcal C[0,1]$ is called strongly degenerate if there exists $x_0\in [0,1]$ such that $a(x_0)=0$, $a(x)>0$ for all $x \in [0,1]\setminus \{x_0\}$ and $\displaystyle\frac{1}{a}\notin L^1(0,1)$. 
\end{Definition}
As an example of a strongly degenerate function $a$ we can take $a(x)=|x-x_0|^K$, with $K \ge 1$.

For any strongly degenerate function $a$ let us introduce the corresponding weighted spaces:
\begin{equation}\label{space}
\begin{aligned}
H^i_a(0,1):=\{u\in H^{i-1}(0,1):& \; u^{(i-1)} \text{ is locally absolutely continuous in}\\
 & [0,1]\setminus \{x_0\} \text { and } \sqrt{a}u^{(i)}\in L^2(0,1)\},
\end{aligned}\end{equation}
equipped with the norms \eqref{normadiv}, $i=1,2$. Also in this case, adding an additional assumption on $a$, the analogous of Proposition \ref{equivalenza1} holds. 
In particular:
\begin{Assumptions}\label{Ipoaggiuntiva}
Assume that there exists $K \in [1,2)$ such that the function 
$
x \mapsto \frac{|x-x_0|^K}{a}
$
is
\begin{enumerate}
\item non increasing on the left of $x_0$ and non decreasing on the right of $x_0$, if $x_0 \in (0,1),$
\item non decreasing on the right of $0$, if $x_0=0,$
\item non increasing on the left of $1$, if $x_0=1$.
\end{enumerate}
\end{Assumptions}
Observe that the previous assumption on a is not surprising because already used
in other papers; moreover this assumption and the requirement $K \ge1$ are satisfied by the prototype that we have in mind $a(x)=|x-x_0|^K$, with $K \ge1$. More precisely, since in the previous hypothesis we require that $K<2$, then as prototype we consider $a(x)=|x-x_0|^K$, with $K \in [1,2).$
 
\begin{Proposition}\label{equivalenza2} Assume Hypothesis \ref{Ipoaggiuntiva}. For  all $u \in H^2_a(0,1)$, set
\[
\|u \|_{2,a}^2:= \|u\|^2_{L^2(0,1)}+ \|\sqrt{a}u''\|^2_{L^2(0,1)}.
\]
Then $\|u\|_{H^2_a(0,1)}$ and $\|u \|_{2,a}$ are equivalent.
\end{Proposition}
\begin{proof} Obviously, there exists a positive constant $C$ such that for all $u \in H^2_a(0,1)$,
\[
\|u \|_{2,a}\le C \|u\|_{H^2_a(0,1)}.
\]
Now, we will prove the other inequality.

Assume, for simplicity, $x_0=0$.
As a first step we prove that there exists a positive constant $C$ such that
\begin{equation}\label{717}
\int_0^1 (v'(x))^2dx \le C\|\sqrt{a}v''\|_{L^2(0,1)}^2,
\end{equation}
for all
\[v \in \mathcal X:= \{ v \in H^2_a(0,1): \; \exists\;  y_0 \in (0,1) \text{ such that } v'(y_0)=0\}.\]

Take $x \in (0,1]$ and $v \in \mathcal X$; then there exists $y_0 \in (0,1)$ such that $v'(y_0)=0$ and
\[
v'(x)=v'(x)-v'(y_0) = \int_{y_0}^x\frac{v''(y)\sqrt{a}}{\sqrt{a}} dy\le \|\sqrt{a}v''\|_{L^2(0,1)} \left(\int_{y_0}^x \frac{1}{a(t)}dt\right)^{\frac{1}{2}}.
\]
Hence
\[
\int_0^1 (v'(x))^2dx \le \|\sqrt{a}v''\|_{L^2(0,1)}^2 \int_0^1\left( \int_{y_0}^x \frac{1}{a(t)}dt\right)dx.
\]
Thus, it is sufficient to estimate the integral in the right hand side of the previous inequality.
To this aim, we divide the integral in the following way:
\[
\int_0^1\left( \int_{y_0}^x \frac{1}{a(t)}dt\right)dx = \int_0^{y_0}\left( \int_{y_0}^x \frac{1}{a(t)}dt\right)dx+ \int_{y_0}^1\left( \int_{y_0}^x \frac{1}{a(t)}dt\right)dx.
\]
Using the assumption on $a$, one has
\[
\int_{y_0}^1\left( \int_{y_0}^x \frac{1}{a(t)}dt\right)dx \le C \int_{y_0}^1\left( \int_{y_0}^1 \frac{1}{t^K}dt\right)dx \le \begin{cases}
C\frac{y_0^{1-K}}{a(1)(K-1)}, &K \neq 1,\\
C\frac{-\log(y_0)}{a(1)}, &K=1,
\end{cases}
\]
for a positive constant $C$.
Now, we consider the term $ \int_0^{y_0}\left( \int_{y_0}^x \frac{1}{a(t)}dt\right)dx$. Using again the assumptions on $a$ and the fact that the constant in Hypothesis \ref{Ipoaggiuntiva} is strictly less than $2$, one has
\[
\begin{aligned}
\left|\int_0^{y_0}\left( \int_{y_0}^x \frac{1}{a(t)}dt\right)dx\right|&=\left|- \int_0^{y_0}\left( \int_x^{y_0} \frac{1}{a(t)}dt\right)dx\right|= \int_0^{y_0}\left( \int_x^{y_0} \frac{1}{a(t)}dt\right)dx\\
& =\int_0^{y_0}\int_0^t\frac{1}{a(t)}dx\, dt =  \int_0^{y_0}\frac{t}{a(t)}dt= \int_0^{y_0}\frac{t^K}{a(t)t^{K-1}}dt\\&
 \le C\int_0^{y_0}\frac{1}{t^{K-1}}dt
\le C\frac{1}{2-K}y_0^{2-K}.
\end{aligned}
\]
Thus,
\[
\int_0^1\left( \int_{y_0}^x \frac{1}{a(t)}dt\right)dx\le C
\]
and
\[
\int_0^1 (v'(x))^2dx \le C\|\sqrt{a}v''\|_{L^2(0,1)}^2,
\]
for a positive constant $C$.
Now, we will prove the thesis for all $u \in H^2_a(0,1)$. To this aim, consider $u\in H^2_a(0,1)$ and let $P$ be the subspace of polynomials of degree one. Then, we can find a polynomial $p_1$ of degree one such that
\[
\|u-p_1\|_{L^2(0,1)}= \min_{p \in P}\|u-p\|_{L^2(0,1)}.
\]
Set
$v:= u -p_1$; then $v \in H^2_a(0,1)$, $v$ has at least two zeros and its derivative vanishes at least once (see Lemma \ref{lemmaequi} below). Hence $v\in \mathcal X$  and, by \eqref{717}, there exists a positive constant $C$ such that
\[
\| v'\|_{L^2(0,1)}^2=\int_0^1 (v'(x))^2dx \le C\|\sqrt{a}v''\|_{L^2(0,1)}^2.
\]
Hence
\begin{equation}\label{718}
\| u'\|_{L^2(0,1)} \le \| v'\|_{L^2(0,1)}  + \| p_1'\|_{L^2(0,1) }\le C \|\sqrt{a}v''\|_{L^2(0,1)}+ \| p_1'\|_{L^2(0,1) }.
\end{equation}
It remains to estimate $\| p_1'\|_{L^2(0,1) }$. To this aim, observe that obviously there exists  a positive constant $C$ such that
\begin{equation}\label{717 bis}
\|p_1'\|_{L^2(0,1)} \le C \|p_1\|_{L^2(0,1)}.
\end{equation}
Moreover
\[
\|v\|_{L^2(0,1)}= \|u-p_1\|_{L^2(0,1)}\le \|u\|_{L^2(0,1)}
\]
and
\begin{equation}\label{719}
\|p_1\|_{L^2(0,1)} \le \|u-p_1\|_{L^2(0,1)}+ \|u\|_{L^2(0,1)} \le 2 \|u\|_{L^2(0,1)}.
\end{equation}
Hence, by \eqref{717 bis} and \eqref{719}, it follows
\begin{equation}\label{720}
\|p_1'\|_{L^2(0,1)} \le C \|u\|_{L^2(0,1)},
\end{equation}
for a positive constant $C$.  By \eqref{718} and \eqref{720}, we obtain
\begin{equation}\label{(2.24) primo}
	\| u'\|_{L^2(0,1)} \le C \|u\|_{2,a},
\end{equation}
thus the thesis follows.

The proof in the case $x_0 \neq 0$ is similar, so we omit it.
\end{proof}
Clearly, the analogous of Proposition \ref{equivalenza1} holds in the case $i=1$  and $a$ strongly degenerate without additional assumption on the function $a$ itself.

\begin{Lemma}\label{lemmaequi}
Let $X:= L^2(0,1)\cap \mathcal{C}[0,1]$ and let $P$ be the subspace of polynomials of degree one. For all $u \in H^2_a(0,1)$, if $p_1\in P$ is such that
\[
\|u-p_1\|_{L^2(0,1)}= \min_{p \in P}\|u-p\|_{L^2(0,1)},
\]
then the function
$v:= u -p_1$ has at least two zeros.
\end{Lemma}
\begin{proof}
Assume that $v(x):= u(x)-p_1(x) \neq 0$ for all $x \in [0,1]$. Without loss of generality, we can assume
$v(x) >0$. Clearly, $v(x) \ge \min_{[0,1]} (u-p_1)  =:\alpha >0$. Assume that $p
_1(x)=mx+q$, $m, q \in \R$. Then $p_1+ \alpha \in P$ and
\[
\|u-p_1\|_{L^2(0,1)} ^2-\|u-p_1-\alpha\|_{L^2(0,1)} ^2=  \alpha \int_0^1(-\alpha+2u-2mx -2q) dx.
\]
Recalling that  $u(x)-mx-q \ge \alpha$ and $\alpha >0$, we have
\[
\|u-p_1\|_{L^2(0,1)} ^2-\|u-p_1-\alpha\|_{L^2(0,1)} ^2
\ge0.
\]
But this is not possible since, by assumption, $p_1$ is such that
\[
\|u-p_1\|_{L^2(0,1)}= \min_{p \in P}\|u-p\|_{L^2(0,1)}.
\]
Now, assume that there exists only a point $y_0 \in [0,1]$ such that $v(y_0)=0$ and $v(x) \neq 0$ for all $x \neq y_0$. For $\epsilon \in \left(0, 2\int_0^1 v(x) dx\right) $, consider $g_\epsilon(x)= v(x) -\epsilon$. Then one can prove 
\[
\|g_\epsilon\|^2_{L^2(0,1)} < \|v\|^2_{L^2(0,1)}.
\]
Indeed
\[
\begin{aligned}
\|g_\epsilon\|^2_{L^2(0,1)}= \int_0^1 ( v^2 + \epsilon^2-2\epsilon v)dx <\int_0^1 v^2 dx 
 &\Leftrightarrow 
\epsilon\left( \epsilon -2\int_0^1 v\, dx\right) <0\\
& \Leftrightarrow \epsilon < 2\int_0^1 v\, dx.
\end{aligned}
\]
Again, we find $p(x):= p_1(x) +\epsilon \in P$, such that
\[
\| u-p\|_{L^2(0,1)} \le \| u-p_1\|_{L^2(0,1)}
\]
and this is not possible. Thus, the thesis follows.
\end{proof}

Also in the strongly degenerate case we consider the space
${\mathcal Z_w}(0,1)$  given in \eqref{Z}, where $H^2_a(0,1)$ is the one defined in \eqref{space}. In order to distinguish the two spaces, we use the notation $\mathcal Z_s(0,1)$ if $a$ is strongly degenerate. Thus, if $u\in {\mathcal Z_s}(0,1)$, $u'$ is locally absolutely continuous in $[0,1]\setminus \{x_0\}$ and not absolutely continuous in $[0,1]$ as for the weakly degenerate case; so equality (\ref{gf0}) is not true a priori. For this reason in \cite{CF_Neumann} we characterize  the space $\mathcal Z_s(0,1)$. In particular, we introduce the space
\begin{equation*}
	\begin{split}
	X:=\{u\in H^1(0,1):\,
	&\text{$u'$ is locally absolutely continuous in $[0,1]\setminus \{x_0\}$},\\
	&au, au' \in H^1(0,1), au''\in H^2(0,1), \sqrt{a}u''\in L^2(0,1),\\
	& (au^{(k)})(x_0)=0, \text{ for all } k=0,1,2\}.
	\end{split}
	\end{equation*}
	Using the definition of $X$ one can easily have the next property.
\begin{Lemma}\label{lemma2.2}[see \cite[Lemma 2.2]{CF_Neumann}] For all $u \in X$ we have that
	\begin{enumerate}
		\item $|a(x)u(x)| \le \|(au)'\|_{L^2(0,1)}\sqrt{|x-x_0|}$,
		\item $|a(x)u'(x)| \le \|(au')'\|_{L^2(0,1)}\sqrt{|x-x_0|}$,
		\item $|a(x)u''(x)| \le \|(au'')'\|_{L^2(0,1)}\sqrt{|x-x_0|}$
	\end{enumerate}
	for all $x \in [0,1]$.
\end{Lemma}
Thanks to the previous estimates and the fact that $\ds\frac{1}{a} \not\in L^1(0,1)$, one can prove the next characterization.
\begin{Proposition}(\cite[Proposition 2.1]{CF_Neumann})\label{Proposition 2.1}
	The spaces $X$ and $\mathcal Z_s(0,1)$ coincide.
\end{Proposition}
\vspace{0,3cm}

For the non divergence case we consider the same spaces as for the weakly degenerate case but, to prove a formula similar to \eqref{gf2}, we have to characterize the space $H^2_{\frac{1}{a}}(0, 1)$.
Thus, we introduce
\begin{equation*}
	Y:=\biggl \{u\in H^2_{\frac{1}{a}}(0, 1): u(x_0)=(au')(x_0)=0 \biggr \}
\end{equation*}
and, proceeding as in \cite{CaFrRo} and \cite{FGGR} (if $x_0 \in \{0,1\}$) or in \cite{CF} (if $x_0 \in (0,1)$), one can prove the following result. 
\begin{Proposition} \label{Prop 3.2}
	If Hypothesis \ref{Ipoaggiuntiva} is satisfied, then
	\begin{equation*}
		H^2_{\frac{1}{a}}(0, 1)=Y.
	\end{equation*}
\end{Proposition}

Hence,  
if Hypothesis \ref{Ipoaggiuntiva} is satisfied, we can rewrite the space $\mathcal W_s(0,1)$ defined as in \eqref{W} in the following way \[
\mathcal W_s (0,1)=\Bigl \{u\in H^2_{\frac{1}{a}}(0, 1):u(x_0)=(au')(x_0)=0 \text{ and } au''''\in L^2_{\frac{1}{a}}(0, 1) \Bigr \}.
\]
As for the weakly degenerate case, one can prove the following Green formulas:
\begin{Lemma}(\cite[Lemmas 2.3, 3.2]{CF_Neumann})\label{Lemma 2.6}
		If $a$ is strongly degenerate, then
	\begin{enumerate}
		\item equality \eqref{gf0} holds for all $(u,v)\in{\mathcal Z_s (0,1)} \times H^2_a(0,1)$;
		\item assume Hypothesis \ref{Ipoaggiuntiva}:
		
 $\star$ if $x_0\in (0,1)$, then for all $(u,v)\in \mathcal W_s (0,1)\times H^2_{\frac{1}{a}}(0, 1)$, equality \[\int_0^1 u''''vdx= [u'''v]^{x=1}_{x=0}-[u''v']^{x=1}_{x=0}+ [u''v']^{x_0^+}_{x_0^-} +\int_0^1 u''v'' dx,\]
holds. Here $u''(x_0^+)= \lim_{\delta\rightarrow 0^+} u''(x_0+\delta)$ and $u''(x_0^-)= \lim_{\delta\rightarrow 0^+} u''(x_0-\delta)$;

$\star$ if $x_0=0$, then for all $(u,v)\in \mathcal W_s (0,1)\times H^2_{\frac{1}{a}}(0, 1)$
		\begin{equation*}
		\int_{0}^{1}u''''v\,dx=u'''(1)v(1)-[u''v']^{x=1}_{x=0}+\int_{0}^{1}u''v''dx;
	\end{equation*}

$\star$	if $x_0=1$, then for all $(u,v)\in \mathcal W_s (0,1)\times H^2_{\frac{1}{a}}(0, 1)$
\begin{equation*}
	\int_{0}^{1}u''''v\,dx=-u'''(0)v(0)-[u''v']^{x=1}_{x=0}+\int_{0}^{1}u''v''dx.
\end{equation*}

\end{enumerate}
\end{Lemma}
Actually, Proposition \ref{Prop 3.2} and Lemma \ref{Lemma 2.6}.$2$ are proved in \cite{CF_Neumann} under a weaker assumption (see \cite[Hypothesis 3.1]{CF_Neumann}; naturally, Hypothesis \ref{Ipoaggiuntiva} implies this assumption). However, here we consider a stronger assumption since we have to apply Lemma \ref{Lemma 2.6}.$2$ under this hypothesis.
\section{The operator in divergence form with generalized Wentzell boundary conditions and interior degeneracy}\label{sezione 3}
Let us fix $\beta_j, \gamma_j \in \mathbb{R}$ such that $\beta_j>0$ and $\gamma_j \le 0$, $j=0,1$. Consider a weakly or a strongly degenerate function $a$ and assume that the degeneracy point $x_0$ is in the {\it interior} of the domain. Let $f: [0,1] \rightarrow \R$ be a continuous function such that
\[
\int_{0}^{1} |f(x)|^2dx+\frac{a(0)|f(0)|^2}{\beta_0}+\frac{a(1)|f(1)|^2}{\beta_1} \in \R
\]
and define $X_\mu$ to be the completion of $\mathcal C[0,1]$ with respect to the norm 
$\|\cdot\|_{X_\mu}$, where
 \begin{equation*}
\norm{f}^2_{X_\mu}=\int_{0}^{1} |f(x)|^2dx+\frac{a(0)|f(0)|^2}{\beta_0}+\frac{a(1)|f(1)|^2}{\beta_1}.
\end{equation*}
By \cite{2} and \cite{4}, it follows that 
\begin{equation*}
X_\mu:= L^2([0,1],d\mu),
\end{equation*}
where
\begin{equation*}
d\mu := dx\restr{(0,1)} \oplus \frac{adS}{\beta}\restr{\{0,1\}},
\end{equation*}
$dx$ denotes the Lebesgue measure on $(0,1)$, $\beta =(\beta_0,\beta_1)$ and $\frac{adS}{\beta}$ denotes the natural Dirac measure $dS$ on $\{0,1\}$ with weight $\frac{a}{\beta}$. More precisely, $X_\mu$ is a Hilbert space with respect to the inner product given by
\begin{equation*}
\left\langle f,g \right\rangle_{X_\mu}=\int_{0}^{1} f(x)g(x)\,dx+\frac{a(0)f(0)g(0)}{\beta_0}+\frac{a(1)f(1)g(1)}{\beta_1},
\end{equation*}
where $f,g\in X_\mu$ are written as $(f\chi_{(0,1)},(f(0),f(1)))$, $(g\chi_{(0,1)},(g(0),g(1)))$. We recall that, as usual, $\chi_{(0,1)}$ denotes the characteristic function of the interval $(0,1)$.
Clearly, $H^i(0,1)\subseteq X_\mu$, $i=0,1,2$, where we recall $H^0(0,1)=L^2(0,1)$.
In particular, $\mathcal C[0,1] \subseteq X_\mu\subseteq L^2(0,1)$.
Now, define the operator in divergence form $A_1u:= (au'')''$ equipped with the following generalized Wentzell boundary conditions
\[
(GWBC)_1 \quad
A_1u(j)+(-1)^{j+1}\frac{\beta_j}{a(j)}(au'')'(j)+\gamma_ju(j)=0,\quad j=0,1,
\]
and the additional boundary conditions
\[
(BC) \quad u''(0)=u''(1)=0.
\]
We distinguish between the weakly degenerate case and the strongly degenerate one.
\vspace{0.5cm}

If $a$ is weakly degenerate we define the weighted space:
\begin{equation*}
\mathcal Z_\mu(0, 1):=\{u \in H^2_{a}(0, 1): au'' \in H^2(0,1) \text{ and } (au'')'' \in X_\mu\}.
\end{equation*}
Clearly $\mathcal Z_\mu (0,1) \subseteq \mathcal Z_w (0,1)$ and, using the definition of the space $H_a^2(0,1)$, $\mathcal Z_\mu  (0, 1)$ can be rewritten also as
\[
\begin{aligned}
\mathcal Z_\mu  (0, 1)& = \{u\in H^1(0,1):\, \text{$u'$ is absolutely continuous in [0,1]},\\ &
\qquad \sqrt{a}u'' \in L^2(0,1), au'' \in H^2(0,1), (au'')'' \in X_\mu\}.
\end{aligned}
\]

	Let us observe that  for any $(u,v)\in \mathcal Z_\mu (0, 1) \times H^2_{a}(0, 1)$ the Green formula \eqref{gf0}  holds. 

Now, let us define the domain of the operator $A_1$ through the following subspace of $Z_\mu \subseteq X_\mu$:
\begin{equation*}
\begin{split}
D_w(A_1):=\biggl \{u\in \mathcal Z_\mu  (0, 1): (GWBC)_1 \text{ and } (BC) \text{ hold}\biggr \}.
\end{split}
\end{equation*}
Hence, we can prove the following theorem:
\begin{Theorem}\label{Teorema generazione caso div debole}
The operator $A_1: D_w(A_1) \rightarrow X_\mu$ is non negative, self-adjoint with dense domain. Thus $-A_1$ generates a contraction semigroup.
\end{Theorem}
\begin{proof}
First of all we prove that $A_1$ is symmetric and non negative on $X_\mu$.
	
	\underline{$A_1$ is symmetric}: take $u, v \in D_w(A_1)$. Then $(u,v) \in \mathcal Z_\mu(0,1) \times H^2_a(0,1)$ and \eqref{gf0} holds. Consequently
	\begin{equation*}
	\begin{split}
	\left\langle A_1u, v \right\rangle_{X_\mu}
	&=\int_{0}^{1} (a u'')''v\,dx + \frac{a(0) (au'')''(0) v(0)}{\beta_0} + \frac{a(1) (au'')''(1) v(1)}{\beta_1} \\
	&=[(au'')'v]^{x=1}_{x=0}-[au''v']^{x=1}_{x=0}+\int_{0}^{1}au''v''dx\\
	&+\frac{a(0)v(0)}{\beta_0}\biggl (\frac{\beta_0}{a(0)}(au'')'(0)-\gamma_0u(0)\biggr )\\
	&-\frac{a(1)v(1)}{\beta_1}\biggl (\frac{\beta_1}{a(1)}(au'')'(1) +\gamma_1u(1)\biggr ) \\
	&=v(1)(au'')'(1)-v(0)(au'')'(0)+\int_{0}^{1}au''v''dx-\frac{\gamma_0}{\beta_0}a(0)v(0)u(0) \\
	&-\frac{\gamma_1}{\beta_1}a(1)v(1)u(1)+v(0)(au'')'(0) - v(1)(au'')'(1) \\
	&= \int_{0}^{1}au''v''dx-\frac{\gamma_0}{\beta_0}a(0)v(0)u(0)-\frac{\gamma_1}{\beta_1}a(1)v(1)u(1)\\
	&=\left\langle u, A_1v \right\rangle_{X_\mu}.
	\end{split}
	\end{equation*}
	\indent
	\underline{$A_1$ is non negative}: for any $u\in D_w(A_1)$, according to the previous calculations, one has
	\begin{equation*}
	\left\langle A_1u, u \right\rangle_{X_\mu}=\int_{0}^{1}a|u''|^2dx-\frac{\gamma_0}{\beta_0}a(0)|u(0)|^2-\frac{\gamma_1}{\beta_1}a(1)|u(1)|^2\ge 0.
	\end{equation*}
	\indent
	Now we prove that $\lambda I+A_1$ is surjective for sufficiently large $\lambda \in \R$. 
	
	\underline{$\lambda I+A_1$ is surjective}: 
	consider the space $H^2_a(0,1)$ with the inner product 
	\[(u,v)_1:= \langle u, v\rangle_{X_\mu} + \langle \sqrt{a}u'', \sqrt{a} v''\rangle_{L^2(0,1)} \quad\,\,\,\,\,\, \forall \; u,v \in H^2_a(0,1),\]
	which induces the norm
	\[
	\|u\|^2_{\circ} := \|u\|^2_{X_\mu} + \|\sqrt{a}u''\|^2_{L^2(0,1)} \quad\,\,\,\,\,\, \forall \; u \in H^2_a(0,1).
	\]
	It is well known that for all $u \in H^1(0,1)$
\begin{equation}\label{stima su modulo di u(x)}
		|u(x)| \le C \|u\|_{H^1(0,1)}, \quad \forall \; x \in [0,1]
\end{equation}
	for a  suitable positive constant $C$. Thus, thanks to \eqref{stima_der} and (\ref{stima su modulo di u(x)}), one can prove that the norm $\|\cdot\|_\circ$ is equivalent to $\|\cdot\|_{H^2_a(0,1)}$.

Moreover,  
$$
H^2_{a}(0,1) \hookrightarrow X_{\mu} \hookrightarrow (H^2_{a}(0,1))^*,
$$
where $(H^2_{a}(0,1))^*$ is the dual space of $H^2_{a}(0,1)$ with respect to $X_{\mu}$. Let $f\in X_{\mu}$ and define $F: H^2_{a}(0,1)\rightarrow \R$
such that
$$
F(v) = \int_0^1 f v\, dx + \frac{a(0) f(0) v(0)}{\beta_0} + \frac{a(1) f(1) v(1)}{\beta_1}\quad\,\,\,\,\,\,\forall \; v\in H^2_a(0,1).
$$
From $H^2_{a}(0,1)\hookrightarrow X_{\mu},$ it follows that $F\in (H^2_{a}(0,1))^*.$ Now, define
\[
L(u,v) := \lambda\int_0^1 uv\,dx + \int_0^1 au''v''dx+ \frac{(\lambda-\gamma_1)a(1)}{\beta_1}u(1)v(1)+  \frac{(\lambda-\gamma_0)a(0)}{\beta_0}u(0)v(0),
\]
for all $u,v \in H^2_a(0,1)$.
Clearly $L(u,v)$ is a continuous bilinear form. Moreover, it is coercive. Indeed, if we take $\lambda > \gamma_i$, for $i=0,1$, then
\[
\begin{aligned}
L(u,u) &= \lambda \int_0^1 u^2dx + \int_0^1 a(u'')^2dx+ \frac{(\lambda-\gamma_1)a(1)}{\beta_1}u^2(1)+  \frac{(\lambda-\gamma_0)a(0)}{\beta_0}u^2(0)\\
&\ge \alpha \|u\|_\circ^2,
\end{aligned}
\]
for all $u \in H^2_a(0,1)$, where $\alpha := \min\{\lambda,1, \lambda-\gamma_0, \lambda-\gamma_1\}$.
As a consequence, by the Lax-Milgram Theorem, there exists a unique $u\in H^2_{a}(0,1)$ such that
$$
L(u,v)= F(v)
$$
for any $v\in H^2_{a}(0,1)$. Hence, setting for simplicity $C_i:= \frac{a(i)}{\beta_i}$, $i=0,1$, the previous equality can be rewritten as
\begin{equation}\label{w1}
\begin{aligned}
&\int_0^1 au'' v''dx + (\lambda-\gamma_0) C_0 u(0)v(0)+ (\lambda-\gamma_1)C_1u(1)v(1)\\
&=\int_0^1 (f-\lambda u) v\, dx + C_0f(0) v(0) + C_1f(1) v(1).
\end{aligned}
\end{equation}
 In particular, \eqref{w1} holds  for all $v\in C^{\infty}_{c}\left(0,1\right) \subseteq  H^2_{a}(0,1)$. 
 Hence
 \[
\int_0^1 au'' v''dx =\int_0^1 (f-\lambda u) v\, dx.
 \]
 This implies that $(au'')'' = f-\lambda u $ a.e. in $(0,1)$.  Hence, using the fact that $f-\lambda u \in X_\mu \subseteq L^2(0,1)$, one has that  $(au'')'' \in X_\mu$ and $au''\in H^2(0,1)$. Thus $u \in \mathcal Z_\mu(0,1)$.

Now, we come back to \eqref{w1}. Using
\eqref{gf0}, \eqref{w1} becomes
\begin{equation}\label{weq2}
\begin{aligned}
&\int_{0}^{1}(au'')''v\,dx-[(au'')'v]^{x=1}_{x=0}+ [au''v']^{x=1}_{x=0}-\gamma_0 C_0u(0)v(0) -\gamma_1 C_1u(1)v(1)\\
& = \int_0^1 (f-\lambda u)v \, dx\, + C_0(f - \lambda u)(0)v(0)+ C_1(f - \lambda u)(1)v(1),
\end{aligned}
\end{equation}
for all $v\in H^2_{a}(0,1)$.
 Thanks to the fact that $(au'')'' = f-\lambda u $ a.e. in $(0,1)$, \eqref{weq2} becomes
\begin{equation}\label{bt}
\begin{aligned}
&-[(au'')'v]^{x=1}_{x=0}+ [au''v']^{x=1}_{x=0}-\gamma_0 C_0u(0)v(0) -\gamma_1 C_1u(1)v(1)\\
& = C_0(f - \lambda u)(0)v(0)+ C_1(f - \lambda u)(1)v(1),
\end{aligned}
\end{equation}
for all $v\in H^2_{a}(0,1)$. Recalling that $a(0)\neq 0$ and $a(1)\neq0$, we have immediately that
\[
u''(0)=u''(1)=0.
\]
Thus \eqref{bt} becomes
\[
\begin{aligned}
&-[(au'')'v](1)+ [(au'')'v](0)-\gamma_0 C_0u(0)v(0) -\gamma_1 C_1u(1)v(1)\\
&=C_0(au'')''(0)v(0)+ C_1(au'')''(1)v(1)
\end{aligned}
\]
for all $v\in H^2_{a}(0,1)$. Hence
\[
A_1u(j)+(-1)^{j+1}\frac{\beta_j}{a(j)}(au'')'(j)+\gamma_ju(j)=0,\quad j=0,1.\]
This implies that $u\in D_w(A_1)$. Thus $\lambda I+A_1$ is surjective for $\lambda$ sufficiently large.

By Theorems \ref{aggiunto} and \ref{densità} (see Section \ref{Section 5 appendix}), we know that $-A_1$ is self-adjoint with upper bound $0$, has dense domain and $(-A_1, D_w(A_1))$ generates a contraction semigroup.
\end{proof}

Thanks to Theorem \ref{Teorema generazione caso div debole}, one has that the problem 
\begin{equation}\label{problema div}
	\begin{cases}
		u_t(t,x)+A_1u(t,x)=h(t,x), &(t,x)\in (0,T) \times (0,1),\\
		A_1u(t,j)+(-1)^{j+1}\frac{\beta_j}{a(j)}(au_{xx})_{x}(t,j)+\gamma_ju(t,j)=0,&t\in (0,T), j=0,1,\\
		u_{xx}(t,0)=u_{xx}(t,1)=0, &t\in (0,T),\\
		u(0,x)=u_0(x),&x\in(0,1),
	\end{cases}
\end{equation}
is well posed in the sense of Theorem \ref{Teorema stime caso div debole} below. As a first step, we recall the following definition.
\begin{Definition}
	If $u_0 \in X_\mu$ and $h\in L^2(0,T;X_\mu)$, a function $u$ is said to be a weak solution of \eqref{problema div} if
	\begin{equation*}
		u\in \mathcal{C}\Bigl([0,T];X_\mu\Bigr )\cap L^2\Bigl (0,T;H^2_{a}(0, 1) \Bigr )
	\end{equation*}
	and
	\begin{equation*}
		\begin{split}
			&	\int_0^1u(T,x)\varphi(T,x)\,dx-\int_0^1u_0(x)\varphi(0,x)\,dx-\int_{(0,T)\times (0,1)}u(t,x)\varphi_t(t,x)\,dx\,dt\\
			&+\frac{a(1)u(T,1)\varphi(T,1)}{\beta_1}-\frac{a(1)u_0(1)\varphi(0,1)}{\beta_1}-\frac{a(1)}{\beta_1}\int_0^Tu(t,1)\varphi_t(t,1)dt\\
			&+\frac{a(0)u(T,0)\varphi(T,0)}{\beta_0}-\frac{a(0)u_0(0)\varphi(0,0)}{\beta_0}-\frac{a(0)}{\beta_0}\int_0^Tu(t,0)\varphi_t(t,0)dt \\
			&=- \int_{(0,T)\times (0,1)}a(x)u_{xx}(t,x)\varphi_{xx}(t,x)\,dx\,dt-\frac{\gamma_1}{\beta_1}\int_0^Ta(1)u(t,1)\varphi(t,1)dt\\
			&-\frac{\gamma_0}{\beta_0}\int_0^Ta(0)u(t,0)\varphi(t,0)dt+\int_{(0,T)\times (0,1)}h(t,x)\varphi(t,x)\,dx\,dt\\
			&+\int_0^T\frac{a(1)h(t,1)\varphi(t,1)}{\beta_1}dt+\int_0^T\frac{a(0)h(t,0)\varphi(t,0)}{\beta_0}dt
		\end{split}
	\end{equation*}
	for all $\varphi\in H^1(0,T;X_\mu)\cap L^2(0,T;H^2_{a}(0, 1))$.
\end{Definition}
\begin{Theorem}\label{Teorema stime caso div debole}
	For all $h \in L^2(0,T;X_\mu)$ and $u_0\in X_\mu$, there exists a unique solution
	\begin{equation*}
		u\in \mathcal{C}\Bigl([0,T];X_\mu\Bigr )\cap L^2\Bigl (0,T;H^2_{a}(0, 1) \Bigr )
	\end{equation*}
	of \eqref{problema div} such that
	\begin{equation}\label{nuova stima forma div}
		\sup_{t\in [0,T]}\Vert u(t)\Vert^2_{X_\mu}+\int_0^T\Vert u(t)\Vert^2_{H^2_{a}(0, 1)}dt\le C_T\Bigl (\Vert u_0\Vert^2_{X_\mu}+\Vert h\Vert^2_{L^2(0,T;X_\mu)} \Bigr )
	\end{equation}
	for some positive constant $C_T$. Moreover, if $h\in W^{1,1}(0,T;X_\mu)$ and $u_0\in D_w(A_1)$, then
	\begin{equation}\label{regol}
		u\in \mathcal{C}^1\Bigl([0,T];X_\mu\Bigr )\cap \mathcal{C}\Bigl ([0,T];D_w(A_1) \Bigr ).
	\end{equation}
\end{Theorem}
\begin{proof}
	The assertion concerning the assumption $u_0\in X_\mu$ and the regularity of the solution $u$ when $u_0\in D_w(A_1)$ is a consequence of the results in \cite{ball}, \cite[Chapter 3, Section 4, Theorem 4.1 and Remark 4.3]{lions magenes} and of \cite[Lemma 4.1.5, Proposition 4.1.6 and Proposition 4.3.9]{cazenave}, \cite[Propositions 3.2 and 3.3]{daprato}. We only need to prove (\ref{nuova stima forma div}). Let us fix $u_0\in D_w(A_1)$ and consider the corresponding weak solution $u$ is in $\mathcal{C}^1\Bigl([0,T];X_\mu\Bigr )\cap \mathcal{C}\Bigl ([0,T];D_w(A_1)\Bigr )$. Now, we multipy the equation of (\ref{problema div}) by $u$ considering the inner product in $X_\mu$; thus
	\begin{equation*}
		\begin{split}
		&\frac{1}{2}\frac{d}{dt}\Vert u(t)\Vert^2_{X_\mu}+\Vert \sqrt{a}u_{xx}(t)\Vert^2_{L^2(0,1)}-\frac{\gamma_0}{\beta_0}a(0)u^2(t,0)-\frac{\gamma_1}{\beta_1}a(1)u^2(t,1)\\
		&\le\frac{1}{2}\Vert u(t)\Vert^2_{X_\mu}+\frac{1}{2}\Vert h(t)\Vert^2_{X_\mu}.
		\end{split}
	\end{equation*}
Hence we deduce that
	\begin{equation}\label{integrating second inequality}
		\begin{split}
			&\frac{d}{dt}\Vert u(t)\Vert^2_{X_\mu}
			\le \frac{d}{dt}\Vert u(t)\Vert^2_{X_\mu}+2\Vert \sqrt{a}u_{xx}(t)\Vert^2_{X_\mu}\\
			&\le \Vert u(t)\Vert^2_{X_\mu}+\Vert h(t)\Vert^2_{X_\mu}.
		\end{split}
	\end{equation}
	By Gronwall's Lemma, for every $t\in [0,T]$ we obtain
	\begin{equation*}
		\begin{split}
			\Vert u(t)\Vert^2_{X_\mu}\le e^T\Biggl (
			&\Vert u_0\Vert^2_{X_\mu}+\Vert h\Vert^2_{L^2(0,T;X_\mu)} \Biggr).
		\end{split}
	\end{equation*}
	Thus, there exists a positive constant $C$ such that
	\begin{equation}\label{integrating second inequality II}
		\begin{split}
			\sup_{t\in [0,T]}\Vert u(t)\Vert^2_{X_\mu}\le C
			\Bigl (\Vert u_0\Vert^2_{X_\mu}+\Vert h\Vert^2_{L^2(0,T;X_\mu)} \Bigr).
		\end{split}
	\end{equation}
	Integrating the second inequality of (\ref{integrating second inequality}) over $(0,T)$ and using (\ref{integrating second inequality II}), we have
	\begin{equation*}
		\begin{split}
			\int_0^T \Vert\sqrt{a}u_{xx}(t)\Vert^2_{L^2(0,1)}dt\le C\Bigl (\Vert u_0\Vert^2_{X_\mu}+\Vert h\Vert^2_{L^2(0,T;X_\mu)} \Bigr ),
		\end{split}
	\end{equation*}
	for a suitable positive constant $C$, and the thesis follows.

Clearly, (\ref{nuova stima forma div}) and (\ref{regol}) hold also if $u_0\in H^2_a(0,1)$, since $D_w(A_1)$ is dense in $H^2_a(0,1)$.
\end{proof}
\vspace{0.5cm}

Now we assume that $a$ is strongly degenerate and we consider the operator $(A_1,D_s(A_1))$. Here $A_1$ is defined as in the weakly degenerate case and $D_s(A_1)$ is defined as $D_w(A_1)$ where we have to consider $H^2_a(0,1)$ defined in (\ref{space}) in place of $H^2_a(0,1)$ defined in (\ref{(2.9) primo}). In particular,
\begin{equation*}
\begin{split}
D_s(A_1):=\biggl \{u\in \mathcal Z_{s, \mu}  (0, 1): (GWBC)_1 \text{ and } (BC) \text{ hold}\biggr \},
\end{split}
\end{equation*}
where 
\[
\begin{aligned}
\mathcal Z_ {s, \mu}(0,1):=\{u\in H^1(0,1): & \; u' \text{ is locally absolutely continuous in} [0,1]\setminus \{x_0\}\\
 & \text { and } \sqrt{a}u''\in L^2(0,1), au''\in H^2(0,1), (au'')'' \in X_\mu\}.
\end{aligned}
\]
Clearly, the Green's formula given in Lemma \ref{Lemma 2.6} for the divergence case still holds.

Furthermore, define
\begin{equation*}
	\begin{split}
		\widetilde{X}:=\{u\in H^1(0,1):\,
		&\text{$u'$ is locally absolutely continuous in $[0,1]\setminus \{x_0\}$},\\
		&au, au' \in H^1(0,1), au''\in H^2(0,1), \sqrt{a}u''\in L^2(0,1),\\
		& (au'')'' \in X_\mu,  (au^{(k)})(x_0)=0, \text{ for all } k=0,1,2\}.
	\end{split}
\end{equation*}
In analogy with Proposition \ref{Proposition 2.1} one has that
\begin{equation*}
	\widetilde{X} = \mathcal Z_{s, \mu}(0,1).
\end{equation*}
As for the weakly degenerate case, one has the next result which contains the generation property in the strongly degenerate context.
\begin{Theorem}
	Assume Hypothesis \ref{Ipoaggiuntiva}. The operator $A_1: D_s(A_1) \rightarrow X_\mu$ is non negative, self-adjoint with dense domain. Thus $-A_1$ generates a contraction semigroup.
\end{Theorem}
The proof of this theorem is analogous to the one of Theorem \ref{Teorema generazione caso div debole}, so we omit it. In any case we underline that, thanks to (\ref{(2.24) primo}) and (\ref{stima su modulo di u(x)}),  the two norms $\|\cdot\|_{\circ}$ and $\|\cdot\|_{H^2_a(0,1)}$ are equivalent. Thus we can prove that $\lambda I+A_1$ is surjective, for $\lambda$ sufficiently large, and the analogous of Theorem \ref{Teorema stime caso div debole} holds for (\ref{problema div}) if $a$ is strongly degenerate.

\section{The operator in non divergence form with generalized Wentzell boundary conditions and interior degeneracy}\label{sezione 4}
As in the  previous section, let us fix $\beta_j$, $\gamma_j\in\mathbb{R}$ such that $\beta_j>0$, $\gamma_j\le 0$, $j=0,1$. Consider a weakly degenerate function $a$ and assume that the degeneracy point $x_0$ belongs to $(0,1)$. Let $f: [0,1] \rightarrow \R$ be a continuous function such that
\[
\int_{0}^{1} \frac{|f(x)|^2}{a}dx+\frac{|f(0)|^2}{\beta_0}+\frac{|f(1)|^2}{\beta_1} \in \R
\]
and define $Y_\mu$ to be the completion of $\mathcal C[0,1]$ with respect to the norm $\|\cdot\|_{Y_\mu}$, where
\begin{equation*}
	\norm{f}^2_{Y_\mu}=\int_{0}^{1} \frac{|f(x)|^2}{a}dx+\frac{|f(0)|^2}{\beta_0}+\frac{|f(1)|^2}{\beta_1}.
\end{equation*}
By \cite{2} and \cite{4}, it follows that 
\begin{equation*}
Y_\mu := L^2_{\frac{1}{a}}([0,1],d\mu),
\end{equation*}
where 
\begin{equation*}
d\mu := \frac{dx}{a}\restr{(0,1)} \oplus \frac{dS}{\beta}\restr{\{0,1\}}.
\end{equation*}
As before, $dx$ denotes the Lebesgue measure on $(0,1)$, $\beta =(\beta_0,\beta_1)$ and $\frac{dS}{\beta}$ denotes the natural Dirac measure $dS$ on $\{0,1\}$ with weight $\frac{1}{\beta}$. In this way $Y_\mu$ becomes a Hilbert space with the inner product given by 
\begin{equation*}
\left\langle f,g \right\rangle_{Y_\mu}=\int_{0}^{1}\frac{f(x)g(x)}{a}\,dx+\frac{f(0)g(0)}{\beta_0}+\frac{f(1)g(1)}{\beta_1},
\end{equation*}
in which $f,g \in Y_\mu$ are written as in Section \ref{sezione 3}.
\newline
\indent
Clearly, $H^i_{\frac{1}{a}}(0,1)\subseteq Y_\mu$, $i=1,2$; in particular, $\mathcal C[0,1] \subseteq Y_\mu \subseteq L^2_{\frac{1}{a}}(0,1)$.

Now we introduce the operator in non divergence form $A_2u:=au''''$ equipped with the general Wentzell boundary conditions
\[
(GWBC)_2 \quad
		A_2u(j)+(-1)^{j+1}\beta_ju'''(j)+\gamma_ju(j)=0,\quad\,\,\, j=0,1,
\]
and the additional boundary conditions $(BC)$. Moreover, we consider the weighted space: 
\begin{equation*}
	\mathcal W_\mu(0, 1):=\Bigl \{u \in H^2_{\frac{1}{a}}(0, 1): au'''' \in Y_\mu\Bigr \}.
\end{equation*}
Now, since $au'''' \in Y_\mu$ implies $au'''' \in L^2_{\frac{1}{a}}(0,1)$, the same considerations made before Lemma \ref{lemma 2.1} hold; in particular, if $u \in 	\mathcal W_\mu(0,1)$, then
$(au^{(k)})(x_0)=0, k=0,1,2,3.$ 
Moreover, observe that, for any $(u,v)\in \mathcal W_\mu (0, 1) \times H^2_{\frac{1}{a}}(0, 1)$, the Green's formula \eqref{gf2} holds and $\mathcal W_\mu (0,1) \subseteq \mathcal W_w (0,1)$.

Now, let us define the domain of the operator $A_2$ through the following subspace of $W_\mu \subseteq Y_\mu$:
\begin{equation*}
	\begin{split}
		D_w(A_2):=\Bigl \{u\in \mathcal W_\mu  (0, 1): (GWBC)_2 \text{ and } (BC) \text{ hold}\Bigr \}.
	\end{split}
\end{equation*}
Hence, we can prove the following result that establishes the main properties of this operator.
\begin{Theorem}\label{teorema generazione semigruppo caso non div}
	If $a$ is weakly degenerate, then the operator $(A_2,D_w(A_2))$ is self-adjoint and non negative on $Y_\mu$ with dense domain. Thus $-A_2$ generates a contraction semigroup.
\end{Theorem}
\begin{proof}
	First of all, we prove the symmetry and non negativity of $A_2$.
	
	\underline{$A_2$ is symmetric}: let $u,v\in D_w(A_2)$; then $(u,v)\in \mathcal{W_\mu}(0,1)\times H^2_{\frac{1}{a}}(0,1)$ and, by (\ref{gf2}), we have
	\begin{equation*}
	\begin{split}
	\left\langle A_2u, v \right\rangle_{Y_\mu}
	&=\int_{0}^{1}\frac{a u'''' v}{a}\,dx + \frac{a(0) u''''(0) v(0)}{\beta_0} + \frac{a(1) u''''(1) v(1)}{\beta_1} \\
	&=[u'''v]^{x=1}_{x=0}-[u''v']^{x=1}_{x=0}+\int_{0}^{1}u''v''dx+\frac{v(0)}{\beta_0}(\beta_0u'''(0)-\gamma_0u(0))\\
&-\frac{v(1)}{\beta_1}\biggl (\beta_1 u'''(1) +\gamma_1u(1)\biggr ) \\
&=\int_{0}^{1}u''v''dx-\frac{\gamma_0}{\beta_0}v(0)u(0)-\frac{\gamma_1}{\beta_1}v(1)u(1) \\
	&=\left\langle u, A_2v \right\rangle_{Y_\mu}.
	\end{split}
	\end{equation*}
	\indent
	\underline{$A_2$ is non negative}: for any $u\in D_w(A_2)$, according to the previous calculations, one has
	\begin{equation*}
	\left\langle A_2u, u \right\rangle_{Y_\mu}= \int_{0}^{1}|u''|^2dx-\frac{\gamma_0}{\beta_0}|u(0)|^2-\frac{\gamma_1}{\beta_1}|u(1)|^2 \ge 0.
	\end{equation*}
	\indent
	Finally, we prove that $\lambda I+A_2$ is surjective for sufficiently large $\lambda \in \R$. 
	
	\underline{$\lambda I+A_2$ is surjective}: consider the space $H^2_{\frac{1}{a}}(0,1)$ with the inner product 
	\[(u,v)_2:= \langle u, v\rangle_{Y_\mu} + \langle u'', v''\rangle_{L^2(0,1)} \quad\,\,\,\,\,\, \forall \; u,v \in H^2_{\frac{1}{a}}(0,1),\]
	which induces the norm
	\[
	\|u\|^2_{\triangle} := \|u\|^2_{Y_\mu} + \|u''\|^2_{L^2(0,1)} \quad\,\,\,\,\,\, \forall \; u \in H^2_{\frac{1}{a}}(0,1).
	\]
Thus, thanks to (\ref{stima su modulo di u(x)}), the norm $\|\cdot\|_\triangle$ is equivalent to  $\norm{\cdot}_{H^2_{\frac{1}{a}}(0, 1)}$.
	
	Moreover,  
	$$
	H^2_{\frac{1}{a}}(0,1) \hookrightarrow Y_{\mu} \hookrightarrow (H^2_{\frac{1}{a}}(0,1))^*,
	$$
	where $(H^2_{\frac{1}{a}}(0,1))^*$ is the dual space of $H^2_{\frac{1}{a}}(0,1)$ with respect to $Y_{\mu}$. Let $f\in Y_{\mu}$ and define $F: H^2_{\frac{1}{a}}(0,1)\rightarrow \R$
	such that
	$$
	F(v) = \int_0^1 \frac{f v}{a}\, dx + \frac{f(0) v(0)}{\beta_0} + \frac{f(1) v(1)}{\beta_1}\quad\,\,\,\,\,\,\forall \; v\in H^2_{\frac{1}{a}}(0,1).
	$$
	From $H^2_{\frac{1}{a}}(0,1)\hookrightarrow Y_{\mu},$ it follows that $F\in (H^2_{\frac{1}{a}}(0,1))^*.$ Now, define
	\[
	M(u,v) := \lambda\int_0^1 \frac{u v}{a}\,dx + \int_0^1 u''v''dx+ \frac{(\lambda-\gamma_1)}{\beta_1}u(1)v(1)+  \frac{(\lambda-\gamma_0)}{\beta_0}u(0)v(0),
	\]
	for all $u,v \in H^2_{\frac{1}{a}}(0,1)$.
	Clearly $M(u,v)$ is a continuous bilinear form. Moreover, it is coercive. Indeed, taking $\lambda > \gamma_i$, for $i=0,1$, and $\delta := \min\{\lambda,1, \lambda-\gamma_0, \lambda-\gamma_1\}$,
	\[
	M(u,u) = \lambda \int_0^1 \frac{u^2}{a}\,dx + \int_0^1 (u'')^2dx+ \frac{(\lambda-\gamma_1)}{\beta_1}u^2(1)+  \frac{(\lambda-\gamma_0)}{\beta_0}u^2(0)\ge \delta \|u\|_\triangle^2,
	\]
    for all $u \in H^2_{\frac{1}{a}}(0,1)$.
	As a consequence, by the Lax-Milgram Theorem, there exists a unique $u\in H^2_{\frac{1}{a}}(0,1)$ such that
	$$
	M(u,v)= F(v)
	$$
	for any $v\in H^2_{\frac{1}{a}}(0,1)$. Hence, the previous equality can be rewritten as
	\begin{equation}\label{w2}
		\begin{aligned}
			&\int_0^1 u'' v''dx + \frac{(\lambda-\gamma_0)}{\beta_0} u(0)v(0)+ \frac{(\lambda-\gamma_1)}{\beta_1}u(1)v(1)\\
			&=\int_0^1 \frac{f-\lambda u}{a}\, v\, dx + \frac{f(0) v(0)}{\beta_0} + \frac{f(1) v(1)}{\beta_1}.
		\end{aligned}
	\end{equation}
	In particular, \eqref{w2} holds  for all $v\in C^{\infty}_{c}\left(0,1\right) \subseteq  H^2_{\frac{1}{a}}(0,1)$. 
	Hence
	\[
	\int_0^1 u'' v''dx =\int_0^1 \frac{f-\lambda u}{a}\, v\, dx.
	\]
	In other words, the second distributional derivative of $u''$ is equal to $\frac{f-\lambda u}{a}$ a.e. in $(0,1)$. Hence, using the fact that $f-\lambda u \in Y_\mu \subseteq L^2_{\frac{1}{a}}(0,1)$, one has that  $au'''' \in Y_\mu$. Thus $u \in \mathcal W_\mu(0,1)$.
	
	Now, we come back to \eqref{w2}. Using
	\eqref{gf2}, \eqref{w2} becomes
	\begin{equation}\label{weq3}
		\begin{aligned}
			&\int_{0}^{1}u''''v\,dx-[u'''v]^{x=1}_{x=0}+ [u''v']^{x=1}_{x=0}-\frac{\gamma_0}{\beta_0}u(0)v(0) -\frac{\gamma_1}{\beta_1}u(1)v(1)\\
			& = \int_0^1 \frac{f-\lambda u}{a}\,v \, dx\, + \frac{(f - \lambda u)(0)v(0)}{\beta_0}+\frac{(f - \lambda u)(1)v(1)}{\beta_1}
		\end{aligned}
	\end{equation}
	for all $v\in H^2_{\frac{1}{a}}(0,1)$.
	Thanks to the fact that $au'''' = f-\lambda u $ a.e. in $(0,1)$, \eqref{weq3} becomes
	\begin{equation}\label{bt2}
		\begin{aligned}
			&-[u'''v]^{x=1}_{x=0}+ [u''v']^{x=1}_{x=0}-\frac{\gamma_0}{\beta_0}u(0)v(0) -\frac{\gamma_1}{\beta_1}u(1)v(1)\\
			& = \frac{(f - \lambda u)(0)}{\beta_0}v(0)+\frac{(f - \lambda u)(1)}{\beta_1}v(1),
		\end{aligned}
	\end{equation}
	for all $v\in H^2_{\frac{1}{a}}(0,1)$. We have immediately that
	\[
	u''(0)=u''(1)=0.
	\]
	Thus \eqref{bt2} becomes
	\[
	-[u'''v](1)+ [u'''v](0)-\frac{\gamma_0}{\beta_0}u(0)v(0) -\frac{\gamma_1}{\beta_1}u(1)v(1)=\frac{(au'''')(0)}{\beta_0}v(0)+ \frac{(au'''')(1)}{\beta_1}v(1)
	\]
	for all $v\in H^2_{\frac{1}{a}}(0,1)$. Hence $u\in H^2_{\frac{1}{a}}(0,1)$, $au''''\in Y_\mu$, u satisfies $(BC)$ and
	\[
	A_2u(j)+(-1)^{j+1}\beta_ju'''(j)+\gamma_ju(j)=0,\quad j=0,1.\]
	This implies that $u\in D_w(A_2)$. In other words $\lambda I+A_2$ is surjective for $\lambda$ sufficiently large.

	Thanks to Theorems \ref{aggiunto} and \ref{densità} (see Section \ref{Section 5 appendix}), we know that $-A_2$ is self-adjoint with upper bound $0$, has dense domain and $(-A_2, D_w(A_2))$ generates a contraction semigroup.
\end{proof}

As a consequence of Theorem \ref{teorema generazione semigruppo caso non div}, one has that the problem
	\begin{equation}\label{problema non div}
		\begin{cases}
			u_t(t,x)+A_2u(t,x)=h(t,x), &(t,x)\in (0,T) \times (0,1),\\
			A_2u(t,j)+(-1)^{j+1}\beta_ju_{xxx}(t,j)+\gamma_ju(t,j)=0,& t\in (0,T), j=0,1,\\
			u_{xx}(t,0)=u_{xx}(t,1)=0,&t\in (0,T),\\
			u(0,x)=u_0(x),&x\in(0,1),
		\end{cases}
	\end{equation}
	is well posed in the following sense.

\begin{Definition}
	If $u_0 \in Y_\mu$ and $h\in L^2(0,T;Y_\mu)$, a function $u$ is said to be a weak solution of \eqref{problema non div} if
	\begin{equation*}
		u\in \mathcal{C}\Bigl([0,T];Y_\mu\Bigr )\cap L^2\Bigl (0,T;H^2_{\frac{1}{a}}(0, 1) \Bigr )
	\end{equation*}
	and
	\begin{equation*}
		\begin{split}
			&	\int_0^1\frac{u(T,x)\varphi(T,x)}{a(x)}\,dx-\int_0^1\frac{u_0(x)\varphi(0,x)}{a(x)}\,dx-\int_{(0,T)\times (0,1)}\frac{u(t,x)\varphi_t(t,x)}{a(x)}\,dx\,dt\\
			&+\frac{u(T,1)\varphi(T,1)}{\beta_1}-\frac{u_0(1)\varphi(0,1)}{\beta_1}-\frac{1}{\beta_1}\int_0^Tu(t,1)\varphi_t(t,1)dt\\
			&+\frac{u(T,0)\varphi(T,0)}{\beta_0}-\frac{u_0(0)\varphi(0,0)}{\beta_0}-\frac{1}{\beta_0}\int_0^Tu(t,0)\varphi_t(t,0)dt\\
			&=- \int_{(0,T)\times (0,1)}u_{xx}(t,x)\varphi_{xx}(t,x)\,dx\,dt-\frac{\gamma_1}{\beta_1}\int_0^Tu(t,1)\varphi(t,1)dt\\
			&-\frac{\gamma_0}{\beta_0}\int_0^Tu(t,0)\varphi(t,0)dt+\int_{(0,T)\times (0,1)}h(t,x)\frac{\varphi(t,x)}{a(x)}\,dx\,dt\\
			&+\int_0^T\frac{h(t,1)\varphi(t,1)}{\beta_1}dt+\int_0^T\frac{h(t,0)\varphi(t,0)}{\beta_0}dt
		\end{split}
	\end{equation*}
	for all $\varphi\in H^1(0,T;Y_\mu)\cap L^2(0,T;H^2_{\frac{1}{a}}(0, 1))$.
\end{Definition}
In particular, the following well posedness theorem holds.
\begin{Theorem}\label{teo4.2}
		For all $h \in L^2(0,T;Y_\mu)$ and $u_0\in Y_\mu$, there exists a unique solution
	\begin{equation*}
		u\in \mathcal{C}\Bigl([0,T];Y_\mu\Bigr )\cap L^2\Bigl (0,T;H^2_{\frac{1}{a}}(0, 1) \Bigr )
	\end{equation*}
	of \eqref{problema non div} such that
	\begin{equation*}
		\sup_{t\in [0,T]}\Vert u(t)\Vert^2_{Y_\mu}+\int_0^T\Vert u(t)\Vert^2_{H^2_{\frac{1}{a}}(0, 1)}dt\le C_T\Bigl (\Vert u_0\Vert^2_{Y_\mu}+\Vert h\Vert^2_{L^2(0,T;L^2_{\frac{1}{a}}(0,1))} \Bigr )
	\end{equation*}
	for some positive constant $C_T$. Moreover, if $h\in W^{1,1}(0,T;L^2_{\frac{1}{a}}(0,1))$ and $u_0\in D_w(A_2)$, then
	\begin{equation*}
		u\in \mathcal{C}^1\Bigl([0,T];Y_\mu\Bigr )\cap \mathcal{C}\Bigl ([0,T];D_w(A_2) \Bigr ).
	\end{equation*}
\end{Theorem}
\vspace{0.5cm} 


\section{Appendix}\label{Section 5 appendix}
In this last section we just give some important results that we need for the previous proofs. These results are well known, anyway we write them in order to make the paper self-contained. 
\begin{Proposition}[see \cite{nagel}, Chapter 2.3, page 90]
A linear operator $A$ on the real Hilbert space $H$ is dissipative if and only if
\[
\langle Au,u \rangle_H \le 0
\]
for all $u \in D(A)$.
\end{Proposition}

\begin{Definition}
A linear operator $A$ is {\it bounded above} in the Hilbert space $H$ if there exists $\omega \in \R$ such that
\[
\langle Au,u \rangle_H \le \omega \langle u, u \rangle_H
\]
for all $u \in D(A)$. In this case $\omega$ is called an upper bound of $A$.
\end{Definition}
Thus, by the previous proposition, we have that if $A$ is dissipative on the real Hilbert space $H$, then it is bounded above with upper bound $0$. Moreover the following theorem holds:
\begin{Theorem}[see \cite{1}, Theorem B.14]\label{aggiunto}
Let $A$ be a linear operator on the Hilbert space $H$ and let $\omega \in \R$. The following assertions are equivalent:
\begin{enumerate}
\item $A$ is self-adjoint with upper bound $\omega$
\item \begin{enumerate} \item $\langle Au,v \rangle_H =\langle v,Au \rangle_H$ for all $u, v \in D(A)$,
\item $\langle Au, u \rangle_H \le \omega \langle u,u \rangle_H$ for all $u \in D(A)$,
\item there exists $\lambda >\omega$ such that $(\lambda I-A)$ is surjective in $H$.
\end{enumerate}
\end{enumerate}
\end{Theorem}
Finally, we recall the following generation results.
\begin{Theorem}[\cite{nagel}, Chapter 2.3, page 91]\label{generatore}
A self-adjoint operator $(A,D(A))$ on a Hilbert space $H$ generates  a strongly continuous semigroup (of self-adjoint operators) if and only if it is bounded above.
\end{Theorem}
\begin{Theorem}[\cite{nagel}, Corollary 3.20]\label{densità}
Let $(A,D(A))$ be a dissipative operator on a reflexive Banach space such that $\lambda I-A$ is surjective for some $\lambda >0$. Then $A$ is densely defined and generates a contraction semigroup.
\end{Theorem}


\begin{thebibliography}{99}
	\bibitem{1}
	W. Arendt, C.J.K. Batty,  M. Hieber, and F. Neubrander, ``Vector-valued Laplace Transforms
	and Cauchy Problems'', Monographs in Mathematics, Vol. {\bf 96}, Birkh\"auser Verlag, Basel, 2001.
	\bibitem{ball}
	J.M. Ball, {\it Strongly continuous semigroups, weak solutions and the variation of constant formula}, Proc. Amer. Math. Soc., Vol. {\bf 63} (1977), 370-373.
	\bibitem{barbero}
	J.F. Barbero, B.A. Ju\'arez-Aubry, J. Margalef-Bentabol, and E.J.S. Villaseñor, {\it Boundary Hilbert spaces and trace operators}, Class. Quantum Gravity, Vol. {\bf 34} (2017), 11 pp.
	\bibitem{daprato}
	A. Bensoussan, G. Da Prato, M.C. Delfour, and S.K. Mitter, ``Representation and Control of Infinite Dimensional Systems'', Second edition. Birkh\"auser Boston, Inc., Boston, MA, 2007.
	\bibitem{brezis}
	H. Brezis, ``Analyse Fonctionnelle, Th\'eorie et Applications'', 2e tirage, Masson 1987.
	\bibitem{popol calsina}
	A. Calsina, and J.Z. Farkas, {\it Steady states in a structured epidemic model with Wentzell boundary conditions}, J. Evol. Equ., Vol. {\bf 12} (2012), 495-512.
	\bibitem{CF}
	A. Camasta, and G. Fragnelli, {\it A degenerate operator in non divergence form}, to appear in Recent 
Advances in Mathematical Analysis, Trends in Mathematics,
https://doi.org/10.1007/978-3-031-20021-2.
	\bibitem{CF_Neumann} A. Camasta, and G. Fragnelli, {\it Degenerate fourth order parabolic equations with Neumann boundary conditions}, submitted for publication, arXiv:2203.02739.
	\bibitem{CaFrRo}
	P. Cannarsa, G. Fragnelli, and D. Rocchetti, {\it Controllability results for a class of one-dimensional degenerate parabolic problems in nondivergence form}, J. Evol. Equ., Vol. {\bf 8} (2008), 583-616.
	\bibitem{cazenave}
	T. Cazenave, A. Haraux, ``An Introduction to Semilinear Evolution Equations, Oxford Lecture Series in Mathematics and its Applications'', Vol. {\bf 13}, Oxford University Press, 1998.
	\bibitem{coclite}
	G.M. Coclite, A. Favini, C.G. Gal, G.R. Goldstein, J.A. Goldstein, E. Obrecht, and S. Romanelli,  {\it The role of Wentzell boundary conditions in linear and nonlinear analysis}, in Advances in nonlinear analysis: Theory, methods and applications, Math. Probl. Eng. Aerosp. Sci., Vol. {\bf 3} (2009), Camb. Sci. Publ., Cambridge, 277--289.
	\bibitem{dappiaggi}
	C. Dappiaggi, H.R.C. Ferreira, and B.A. Ju\'arez-Aubry, {\it Mode solutions for a Klein-Gordon field in anti-de Sitter spacetime with dynamical boundary conditions of Wentzell type}, Phys. Rev. D, Vol. {\bf 97} (2018), 12 pp.
	\bibitem{duvaut}
	G. Duvaut, and J.L. Lions, ``Inequalities in Mechanics and Physics'', Springer-Verlag, Heidelberg, Berlin, Germany, 1976.
	\bibitem{nagel}
	K.J. Engel, and R. Nagel, ``One-Parameter Semigroups for Linear Evolution Equations'', Springer-Verlag, New York, 2000.
	\bibitem{popol farkas}
	J.Z. Farkas, and P. Hinow, {\it Physiologically structured populations with diffusion and dynamic boundary conditions}, Math. Biosci. Eng., Vol. {\bf 8} (2011), 503-513.
	\bibitem{2}
	A. Favini, G.R. Goldstein, J.A. Goldstein, and S. Romanelli,  {\it The heat equation with generalized Wentzell boundary conditions}, J. Evol. Equ., Vol. {\bf 2} (2002), 1--19.
	\bibitem{favini vol spec}
	A. Favini, G.R. Goldstein, J.A. Goldstein, and S. Romanelli,  {\it Fourth order ordinary differential operators with general Wentzell boundary conditions}, Differential Equations: Inverse and Direct Problems (2006), 59--72.
	\bibitem{feller}
	W. Feller, {\it The parabolic differential equations and the associated semi-groups of transformations}, Ann. Math., Vol. {\bf 55} (1952), 468-519. 
	\bibitem{3}
	G. Fragnelli, J.A. Goldstein, R.M. Mininni, and S. Romanelli, {\it Operators of order $2n$ with interior degeneracy}, Discrete and Continuous Dynamical Systems, Series S, Vol. {\bf 13} (2020), 3417-3426.
	\bibitem{4}
	G. Fragnelli, G.R. Goldstein, J.A. Goldstein, R.M. Mininni, and S. Romanelli, {\it Generalized Wentzell boundary conditions for second order operators with interior degeneracy}, Discrete Contin. Dyn. Syst.-S, Vol. {\bf 9} (2016), 697-715.
	\bibitem{FGGR}
	G. Fragnelli, G.R. Goldstein, J.A. Goldstein, and S. Romanelli, {\it Generators with interior degeneracy on spaces of $L^2$ type}, Electron. J. Diff. Equ., Vol. {\bf 2012} (2012), 1--30.
	\bibitem{6}
	G. Fragnelli, and D. Mugnai, ``Control of Degenerate and Singular Parabolic Equations. Carleman Estimates and Observability'', SpringerBriefs in Mathematics, Springer International Publishing, 2021.
	\bibitem{gal}
	C.G. Gal, {\it On a nonlinear variant of the beam equation with Wentzell boundary conditions}, Differ. Integral Equ., Vol. {\bf 18} (2005), 71--81.
	\bibitem{gal 1}
	C.G. Gal, {\it A Cahn-Hilliard model in bounded domains with permeable walls}, Math. Meth. Appl. Sci., Vol. {\bf 29} (2006), 2009--2036.
	\bibitem{gal 2}
	C.G. Gal, {\it Global well-posedness of a non-isothermal Cahn-Hilliard model with dynamic boundary conditions}, Adv. Differ. Equ., Vol. {\bf 12} (2007), 1241--1274.
	\bibitem{7}
	G.R. Goldstein, {\it Derivation and physical interpretation of general Wentzell boundary conditions}, Adv. Differ. Equ., Vol. {\bf 11} (2006), 457--480.
	\bibitem{leoni}
	G. Leoni, ``A First Course in Sobolev Spaces'', Second Edition, Graduate Studies in Mathematics, Vol. {\bf 181}, American Mathematical Society, 2019.
	\bibitem{lions magenes}
	J.L. Lions, and E. Magenes, ``Non-Homogeneous Boundary Value Problems and Applications'', Volume I, Translated from the French by P. Kenneth, Grundlehren Math. Wiss. 181, Springer, New York, 1972.
	\bibitem{popol mokhtar}
	M. Mokhtar-Kharroubi, {\it Spectra of structured diffusive population equations with generalized Wentzell-Robin boundary conditions and related topics}, Discrete Contin. Dyn. Syst.-S, Vol. {\bf 13} (2020), 3551-3563.
	\bibitem{popol mokhtar richard}
	M. Mokhtar-Kharroubi, and Q. Richard, {\it Time asymptotics of structured populations with diffusion and dynamic boundary conditions}, Discrete Continuous Dyn. Syst. Ser. B, Vol. {\bf 23} (2018), 4087-4116.
	\bibitem{mugnolo}
	D. Mugnolo, and S. Romanelli, {\it Dirichlet forms for general Wentzell boundary conditions, analytic semigroups, and cosine operator functions}, Electron. J. Differ. Equ., Vol. {\bf 118} (2006), 1-20.
	\bibitem{vogt}
	H. Vogt, and J. Voigt, {\it Wentzell boundary conditions in the context of Dirichlet forms}, Adv. Differ. Equ., Vol. {\bf 8} (2003), 821-842.
	\bibitem{zhan}
	J. Zahn, {\it Generalized Wentzell boundary conditions and quantum field theory}, Ann. Henri Poincar\'e, Vol. {\bf 19} (2018), 163-187.
	\bibitem{wentzell}
	A. D. Wentzell, {\it On boundary conditions for multidimensional diffusion processes}, Theory Probab. Appl., Vol. {\bf 4} (1959), 164-177.
 \end{thebibliography}
\end{document}